\documentclass{amsart}

\usepackage{amsmath}
\usepackage{amssymb}
\usepackage{amsfonts}
\usepackage{enumerate}
\usepackage{vmargin}

\newcommand{\bC}{{\mathbb{C}}}

\newcommand{\bM}{{\mathbb{M}}}

\newcommand{\bR}{{\mathbb{R}}}

  \newcommand{\E}{{\mathcal{E}}}

\renewcommand{\phi}{\varphi}
\newcommand{\upchi}{{\raise.35ex\hbox{\ensuremath{\chi}}}}

\renewcommand{\leq}{\leqslant}
\renewcommand{\geq}{\geqslant}

\newtheorem{thm}{Theorem}[section]

\newtheorem{lemma}[thm]{Lemma}
\newtheorem{rk}[thm]{Remark}
\title{An extension of inequalities by Ando}
\date{}
\author[\'E. Ricard]{\'Eric Ricard}
\address{Normandie Univ, UNICAEN, CNRS, Laboratoire de Math{\'e}matiques Nicolas Oresme, 14000 Caen, France}
\email{eric.ricard@unicaen.fr}

\thanks{{\it 2010 Mathematics Subject Classification:} 46L51; 47A30.} 
\thanks{{\it Key words:} Norm inequalities, functional calculus}
\begin{document}
\begin{abstract}
We give variations on Ando's result comparing $f(B)-f(A)$
and $f(|B-A|)$ with respect to unitarily invariant norms on matrices. 
\end{abstract}

\maketitle

 This note deals with norm matricial inequalities. Our starting point
 is the following inequality by Ando in \cite{A}: if $A, B$ are
 positive matrices and $\|.\|$ is a unitarily invariant norm and $f$
 is an operator monotone function on $\bR^+$ with $f(0)\geq 0$, then
\begin{equation}\label{ando}
\| f(B)-f(A)\| \leq \|f(|B-A|)\|.\end{equation}
 The result was also obtained by
 Birman Koplienko and Solomyak in \cite{BKS}. The inequality reverses
 if the reciprocal of $f$ is operator monotone, this holds for
 instance if $f$ is an increasing operator convex function with
 $f(0)=0$.

 In \cite{R}, it is shown that for any $p\geq 2$, 
$${\rm Tr } (B-A)(B^{p-1}-A^{p-1}) \geq {\rm Tr } |B-A|^{p}.$$ This
 has been recently extended in \cite{DHLV} by Dinh, Ho, Le and Vo to
 any operator convex function $f$ with $f(0)=0$:
$${\rm Tr } (B-A)(f(B)-f(A)) \geq {\rm Tr } |B-A|f(|B-A|).$$ The
 inequality is reversed if $f$ is non-negative operator monotone. It is
 naturally tempting to imagine that for any positive operator monotone
 function one has
$$\|(B-A) (f(B)-f(A))\| \leq \||B-A| f(|B-A|)\|,$$ with reversed
 inequality for positive operator convex functions.  The aim of this
 note is to show that such inequalities hold. To do so, we revisit
 Ando's argument to see how it can be extended.

In the first section, we review basic facts on various comparisons of
matrices before using them to deduce the main inequalities. We assume
that the reader is familiar with matricial inequalities. We have
chosen to stick with matricial inequalities but most of what is done
here can be adapted to general semifinite von Neumann algebras.

\section{Comparisons of matrices}

 We refer \cite{Bha} for basic background on matricial inequalities.
As usual,  $\bM_n$ is the space of matrices of size $n$ over $\bC$ with its 
usual trace ${\rm Tr}$. We denote by $\bM_n^+$ its subset of positive semidefinite matrices.

 Given $A\in \bM_n$, we denote by $s_i(A)$ its singular values in
 decreasing order. We will frequently use that if $A\in \bM_n^+$ and
 $f: \bR^+=[0,\infty)\to \bR^+$ is a positive non-decreasing function then
 $s_i(f(A))=f(s_i(A))$.

We recall classical orders beyond the usual one $\leq$ on selfadjoint matrices.

First for $A,\,B\in \bM_n$, we write $A\preceq B$ if for all $1\leq k\leq n$, 
$\sum_{i=1}^k s_i(A)\leq \sum_{i=1}^k s_i(B)$.

If we set $\|A\|_{(k)}= \sum_{i=1}^k s_i(A)$ for the Ky Fan norms, Ky
Fan's principle (Theorem IV.2.2 in \cite{Bha}) gives that if $A\preceq
B$ iff $\|A\|\leq \|B\|$ for any unitarily invariant norm. This is
also equivalent to the existence of a completely positive map $T:
\bM_n\to \bM_n$ with $T(1)\leq 1$ and ${\rm Tr} \circ T\leq {\rm Tr}$
with $T(|B|)=|A|$ see \cite{H}. Using the polar decomposition, we obtain
that $A\preceq B$ iff there is a map $T:
\bM_n\to \bM_n$, which is contractive for all unitarily invariant norms, so that
$T(B)=A$.
We won't use it but we recall that if
$\phi:\bR^{+}\to \bR^{+}$ is a non-decreasing convex function and
$A,B\in \bM_n^+$ then $\phi(A)\preceq \phi(B)$ if $A\preceq B$.

Finally, for $A\in \bM_n$ and $B\in \bM_N$ with $N\geq n$, we write
$A\ll B$ if for all $1\leq k\leq n$, $ s_k(A)\leq s_k(B)$.  Weyl's
monotonicity principle, Corollary III.2.3 in \cite{Bha}, gives that
for $A,B\in \bM_n^+$ with $A\leq B$, $A\ll B$.  Thus,
using the polar decomposition and diagonalization, it is easy to see
that for $A\in \bM_n$ and $B\in \bM_N,$ $A\ll B$ iff there  are
contractions $C, C'\in \bM_{N,n}$ so that $|A|=C^*|B|C$ and $A=C'^*BC$. Of course
for $A,B\in \bM_n$, $A\ll B$ implies $A\preceq B$.

Now we gather some facts about these comparisons. They must be folklore but we give a proof for completeness. For $A,B\in \bM_n$, we write 
$A\oplus B$ for $\left[\begin{array}{cc} A & 0 \\0& B\end{array}\right]\in \bM_{2n}$.

\begin{lemma}\label{a-b}
If $A,\,B\in \bM_n^+$, then $B-A\ll B\oplus A$.
\end{lemma}
\begin{proof}
Since $B-A$ is selfadjoint, it can be written as 
$B-A=D_+-D_-$ where $D_{\pm }\in \bM_n^+$ and $D_+D_-=0$. 
It follows 
that for any $1\leq k\leq n$, $s_k(B-A)=s_k(D_+\oplus D_-)$. 
Let $e$ and $f$ be the support projections of $D_+$ and $D_-$, then 
$$0\leq D_+ \oplus D_-=
 e(B-A)e \oplus f(A-B)f\leq  eBe \oplus fAf.$$
The result then follows by Weyl's monotony principle as $e\oplus f$ is a contraction.
\end{proof}

\begin{lemma}\label{sum}
Let $D\in \bM_n^+$ and $g_i:\bR^+\to \bR^+$ be  non-decreasing
functions and $A_i\in \bM_n$ for $1\leq i\leq d$ such that $A_i\preceq g_i(D)$, then  
$\sum_{i=1}^dA_i\preceq (\sum_{i=1}^dg_i)(D)$. 
\end{lemma}
\begin{proof}
This is just the triangular inequality for the norms $\|.\|_{(k)}$ combined with the fact that $\sum_{i} \|g_i(D)\|_{(k)}=\|\sum g_i(D)\|_{(k)}$. Indeed the $g_i$ 
are non-decreasing so it yields that $s_j(\sum_i g_i(D))=\sum_is_j( g_i(D))$ for any $1\leq j\leq n$. 
\end{proof}
\begin{lemma}\label{sweyl}
Let $D\in \bM_n^+$ and $g_i:\bR^+\to \bR^+$ be non-decreasing
functions and $A_i\in \bM_n$ so that $A_i\preceq g_i(D)$ for $1\leq
i\leq d$, then $\prod_{i=1}^dA_i\preceq (\prod_{i=1}^dg_i)(D)$.
\end{lemma}

\begin{proof}
  By induction, it suffices to do it for $d=2$.
  We have $A_i=T_i(g_i(D))$ for some  map
  $T_i:\bM_n\to \bM_n$  which is
  contractive for all unitarily invariant norms.
  Let $P_j=1_{[s_j(D),\infty)}(D)$, since $g_i$ is non-decreasing 
there are positive reals $a_{i,j}$ so that $g_i(D)=\sum_{j=1}^n
  a_{i,j}P_j$. Hence $A_1A_2=\sum_{j_1,j_2}
  a_{1,j_1}a_{2,j_2} T_1(P_{j_1})T_2(P_{j_2})$. But $T_i(P_{j_i})$'s are contractions and
  $$\|T_1(P_{j_1})T_2(P_{j_2})\|_{(k)}\leq \min \{
  \|T_1(P_{j_1})\|_{(k)},\|T_2(P_{j_2})\|_{(k)}\} \leq  \min \{
  \|P_{j_1}\|_{(k)},\|P_{j_2}\|_{(k)}\}=\|P_{j_1}P_{j_2}\|_{(k)}.$$
  This means that $T_1(P_{j_1})T_2(P_{j_2})\preceq P_{j_1}P_{j_2}$. 
  Noticing that $P_{j_1}P_{j_2}$  is a non-decreasing function of $D$
and $(g_1g_2)(D)=\sum_{j_1,...,j_d} a_{1,j_1}a_{2,j_2} P_{j_1}P_{j_2}$, we get the conclusion by Lemma \ref{sum}.
\end{proof}

Given a non-commutative polynomial in several variables 
$$P(X_1,...,X_d)=\sum_{l=0}^m \sum_{i_1,...,i_l=1}^d
\alpha_{i_1,...,i_l} X_{i_1}...X_{i_l},$$ 
we define $|P|$ as 
$$|P|(X_1,...,X_d)=\sum_{l=0}^m \sum_{i_1,...,i_l=1}^d
|\alpha_{i_1,...,i_l}| X_{i_1}...X_{i_l}.$$

\begin{lemma}\label{poly}
Let $D\in \bM_n^+$ and $g_i:\bR^+\to \bR^+$ be   non-decreasing
functions and $A_i\in \bM_n$ such that $A_i\preceq g_i(D)$ for $1\leq i\leq d$, then 
for any non-commutative polynomial $P$ of $d$-variables, we have 
$$P(A_1,...,A_d)\preceq |P|(g_1(D),...,g_d(D)).$$
\end{lemma}

\begin{proof}
This is just the combination of Lemmas \ref{sum} and \ref{sweyl}. 
\end{proof}

\begin{rk}\label{genpol}
One can extend the above lemmas in many ways. For instance, we can
assume that we have a continuous sets of variables $(A_i(s))$ and
replace sums $\sum_{i_1,...,i_l=1}^d$ by integrals against positive measures as long as the objects make sense.
\end{rk}

\section{Main inequalities}

First we rewrite Ando's proof from \cite{A} (see also Section X in
\cite{Bha}). We fix $s>0$ and consider the function on $\bR^+$,
$f_s(t)=\frac {t}{s+t}=1-\frac s{s+t}$.  For convenience, we set
$f_0(t)=t$. These are the basic bricks for operator monotone
functions.

\begin{lemma}\label{withD}
Let $A,\,D\in \bM_n^+$, then $f_s(A+D)-f_s(A)\ll  f_s(D)$ for any $s\geq0$.
\end{lemma}

\begin{proof}
This is obvious if $s=0$, we assume $s>0$. 

First we have the identity, $f_s(A+D)-f_s(A)= s \big((A+s)^{-1}- (A+D+s)^{-1}\big)$. With $C=s^{1/2}(s+A)^{-1/2}$, which is a contraction, we have 
$f_s(A+D)-f_s(A)=C f_s(CDC)C$.
It follows that 
$f_s(A+D)-f_s(A)\ll f_s(CDC)$. But
$CDC$ is unitarily equivalent to $0\leq D^{1/2}C^2D^{1/2}
\leq D$. As $f_s$ is operator monotone, we end up with 
$f_s(A+D)-f_s(A)\ll f_s(D^{1/2} C D^{1/2})\leq f_s(D)$.
\end{proof}

\begin{lemma}\label{fsll}
Let $A,\,B\in \bM_n^+$, then $f_s(B)-f_s(A)\ll  f_s(|B-A|)$ for all $s\geq0$ .
\end{lemma}

\begin{proof}
Put $D=B-A$ and define $D_\pm$ as above. Then $A+D_+=B+D_-$ and by
operator monotony of $f_s$ and Lemma \ref{a-b} since $f_s(B)-f_s(A)= f_s(B)-f_s(B+D_-)+f_s(A+D_+)-f_s(A)$, 
$$f_s(B)-f_s(A)\ll\big(f_s(B+D_-)- f_s(B)\big)\oplus \big(f_s(A+D_+)-f_s(A)\big).$$
Thanks to Lemma \ref{withD}, we get $f_s(B)-f_s(A)\ll f_s(D_-)\oplus  f_s(D_+).$

But as $D_-$ and $D_+$ commute and have disjoint supports, 
$f_s(D_-)\oplus f_s(D_+)$ and 
$f_s(|D|) \oplus 0$ are unitarily
equivalent in $\bM_{2n}$ and we can conclude.
\end{proof}

 An operator monotone function $g$ with $g(0)=0$ has an integral
 representation $g(t)=\int_{\bR^+}f_s(t) d\mu(s),$ for some
 positive measure $\mu$ (that may charge $0$) such that $\int_{\bR^+}\frac1{1+s} d\mu(s)<\infty$. Thus Ando's result,
 $g(B)-g(A)\preceq |B-A|$ if $A,B\in \bM_n^+$ follows from Lemma
 \ref{fsll}, as $f_s(B)-f_s(A)\ll f_s(|B-A|)$ and the extension of
 Lemma \ref{sum} to integrals.

By Lemma \ref{poly}, we directly get

\begin{thm}
Let $d,e\geq 1$ and $g_i:\bR^+\to \bR^+$ be operator monotone
functions with $g_i(0)=0$ for $1\leq i\leq d$ and $h_i:\bR^+\to \bR^+$ be
non-decreasing. Then for $P$ a non-commutative polynomial of $d+e$ variables and
any $A,\,B\in \bM_n^+$ and any matrices $C_i$ so that $C_i\preceq h_i(|A-B|)$
for $1\leq i\leq e$, we have
\begin{eqnarray*}P\big(g_1(B)-g_1(A),...,g_d(A)-g_d(B),C_1,...,C_e\big)\hspace{3cm}\\\hspace{1cm} \preceq
|P|\big(g_1(|B-A|),...,g_d(|A-B|), h_1(|A-B|),..., h_e(|A-B|)\big).
\end{eqnarray*}
 
\end{thm}

\begin{rk}
One can see directly in the proof that one just need to assume
$g_i(0)\geq 0$. This can also be seen as applying Lemma \ref{poly} one
more time, as $(g_i-g_i(0))(|B-A|)\ll g_i(|B-A|)$.
\end{rk}

The above theorem can be extended to more general objects other than
polynomials and contains many particular cases. We give a few
examples, assuming that $(g_i)_{i\geq 1}$ are operator monotone
functions with $g_i(0)\geq 0$ and $h_i$ are non-decreasing functions with $h(0)\geq 0$. For any unitarily invariant norm and any $d$, we
have:
\begin{eqnarray}
\|\prod_{i=1}^d(g_i(B)-g_i(A)) \|&\leq& \|\,\prod_{i=1}^dg_i(|B-A|)\|, \\
\| (A-B) (g_1(A)-g_1(B))\|&\leq &\|\, |A-B|g_1(|A-B|)\|,\\\label{hoa}
\| \sum_{i=1}^d h_i(|A-B|) (g_i(A)-g_i(B))\|&\leq &\|\, \sum_{i=1}^d h_ig_i(|A-B|)\|,\\\label{hoa2}
\|(A-B)\exp\big(g_1(A)-g_1(B)\big)\|&\leq&\|\, (A-B)\exp\big(g_1(|A-B|)\big)\|.
\end{eqnarray}

Recall that if $f$ is operator convex on $[0,\infty)$ with $f(0)\leq0$
  then $t\mapsto f(t)/t$ is operator monotone on $(0,\infty)$ by \cite{HP}. In
  particular, if $f$ is non-negative and $f(0)=0$, then $f(t)=tg(t)$
  for $t\in \bR^+$ where $g$ is operator monotone. Thus a non-negative
  operator convex function $f$ on $\bR^+$ with $f(0)=0$ has an
  integral representation:
$$f(t)=\beta t+ \gamma t^2+ \int_{\bR^+} tf_s(t) d\mu(s),$$
for some (positive) measure $\mu$ (that does not charge $0$) such that 
$\int_0^\infty\frac1{1+s} d\mu(s)<\infty$ and  some $\beta,\gamma\geq 0$.

From those inequalities, one can also get results for operator convex
functions, we give one example.

\begin{thm}
Let $f$ be a non-negative operator convex function on $\bR^+$ with $f(0)=0$ and
$h$ a non-decreasing function with $h(0)\geq 0$, then for any
$A,\,B\in \bM_n^+$, we have
$$ hf(|B-A|) \preceq h(|B-A|)\big(f(B)-f(A)\big).$$
\end{thm}

\begin{proof}
We first prove it in the case where $f(t)=\beta t+ \gamma t^2+ \int_0^M tf_s(t)
d\mu(s)$ for some $M\in \bR^+$, by assumption $\beta, \gamma\geq 0$. Note that $tf_s(t)= t-sf_s(t)$ for
$s>0$. It follows that $f(t)=\gamma t^2+ \delta t - \int_0^M sf_s(t)
d\mu(s)$ for some $\delta\geq 0$. The function $g(t)=\int_0^M sf_s(t)
d\mu(s)$ is operator monotone on $\bR^+$ with $g(0)=0$.
By the triangular inequality for any $n\geq k\geq 1$, with $D=B-A$:
$$\|h(|A-B|) \big(f(B)-f(A)\big)\|_{(k)} \geq \| h(|D|) \big(\gamma
(B^2-A^2)+\delta D\big)\|_{(k)}- \| h(|D|)
\big(g(B)-g(A)\big)\|_{(k)}. $$ Let $\E$ be the trace preserving
conditional expectation onto the (commutative) algebra generated by
$D=D_+-D_-$, we have $\E((A+D)^2-A^2)=2\E(A)D+D^2$. If we denote by $p$
and $q$ be the support projections of $D_+$ and $D_-$. As $A\geq 0$,
$p\E (A)\geq0$ and since $A+D_+=B+D_-$, we get $q\E(A)\geq D_-$. Thus,
$p(2\E(A)D+D^2)\geq D_+^2$ and $q(2\E(A)D+D^2)\leq -D_-^2$. Hence we
arrive at $|\E\big(\gamma ((A+D)^2-A^2)+\delta D\big)|\geq \gamma D^2
+ \delta |D|$, from which for any $1\leq k\leq n$
$$\| h(|D|) \big(\gamma
((A+D)^2-A^2)+\delta D\big)\|_{(k)}\geq \| h(|D|)\big(\gamma D^2 + \delta |D|\big)\|_{(k)}.$$
Using inequality \eqref{hoa},
$$\|h(|D|) \big(f(B)-f(A)\big)\|_{(k)} \geq\| h(|D|)\big(\gamma D^2 + \delta |D|\big)\|_{(k)}- \| h(|D|) g(|D|)\|_{(k)}.$$
As for any $1\leq i\leq k$, we have $s_i\big(h(|D|)\big(\gamma D^2 + \delta |D|\big)\big)-s_i(h(|D|) g(|D|))=s_i(h(D))s_i(f(|D|))$, we get
$$\|h(|B-A|) \big(f(B)-f(A)\big)\|_{(k)} \geq \| h(|B-A|) f(|B-A|)\|_{(k)}.$$
The case of general $f$ follows by approximation.
\end{proof}

One can also adapt the arguments to get trace inequalities as in \cite{Hoa2}.
One gets for instance from \eqref{hoa} that if $h:\bR\to \bR$ is an odd or even function non-decreasing on $\bR^+$ with $h(0)=0$ and $g:\bR^+\to \bR^+$ is operator monotone, then for all $A, B\in \bM_n^+$:
$$\Big| {\rm Tr}\,  h(B-A) (g(B)-g(A)) \Big|  \leq    {\rm Tr} \,hg(|B-A|).$$
 The above arguments also give that if  $h:\bR\to \bR$ is an odd function non-decreasing on $\bR^+$ and $f:\bR^+\to \bR^+$ is non-negative operator convex function with $f(0)=0$, then for all $A, B\in \bM_n^+$:
$$ {\rm Tr}\,  h(B-A) (f(B)-f(A))   \geq    {\rm Tr} \,hf(|B-A|).$$

We would like to remark that all of the above inequalities can be
generalized to bounded operators with finite support 
on a semifinite von Neumann algebra 
(that one can assume to be a factor). One
has to use the generalized $s$-numbers of \cite{FK} instead of the
singular values and symmetric function spaces instead of unitarily
invariant norms see \cite{PX}. We leave other possible technical
extensions to the interested readers.

We conclude by noticing that \eqref{ando} does not hold for general
concave functions. For instance, it is false for $f(t)=\min\{t,1\}$
and the operator or the trace norms. Indeed, \eqref{ando} for the
operator norm would imply that $f$ is Lipschitz in that norm.
 By homogeneity and translation, this would imply that the
absolute value is also Lipschitz in the operator norm on selfadjoint operators which is
false.
\smallskip

{\noindent \bf Acknowledgments} The author wishes to thanks Trung Hoa Dinh for exchanges that motivated this note.

\bibliographystyle{plain}
\bibliography{bibli}

\end{document}